\documentclass[12pt,a4paper]{amsart}

\usepackage[all]{xy}
\usepackage{amssymb, amscd}

\theoremstyle{plain}
\newtheorem{theorem}{Theorem}[section]
\newtheorem{lemma}[theorem]{Lemma}
\newtheorem{proposition}[theorem]{Proposition}
\newtheorem{corollary}[theorem]{Corollary}

\theoremstyle{definition}

\newtheorem{remark}[theorem]{Remark}

\newtheorem{example}[theorem]{Example}

\numberwithin{equation}{section}

\newcommand\bG{{\mathbb G}}

\newcommand\bZ{{\mathbb Z}}

\newcommand\bmu{{\mathbb\mu}}

\newcommand\cF{{\mathcal F}}
\newcommand\cG{{\mathcal G}}

\newcommand\cL{{\mathcal L}}

\newcommand\cO{{\mathcal O}}
\newcommand\cP{{\mathcal P}}

\newcommand\wA{\widehat{A}}
\newcommand\wT{\widehat{T}}

\newcommand\aff{\operatorname{aff}}

\newcommand\ant{\operatorname{ant}}
\newcommand\charc{\operatorname{char}}

\newcommand\id{\operatorname{id}}

\newcommand\modf{\operatorname{mod}}

\newcommand\Sym{\operatorname{S}}
\newcommand\V{\operatorname{V}}

\newcommand\Coh{\operatorname{Coh}}
\newcommand\Coker{\operatorname{Coker}}
\newcommand\End{\operatorname{End}}
\newcommand\Ext{\operatorname{Ext}}

\newcommand\Hom{\operatorname{Hom}}
\newcommand\Ima{\operatorname{Im}}
\newcommand\Kern{\operatorname{Ker}}

\newcommand\Mod{\operatorname{Mod}}

\newcommand\QCoh{\operatorname{QCoh}}
\newcommand\Spec{\operatorname{Spec}}

\newcommand\Uni{\operatorname{Uni}}

\title{The coherent cohomology ring of an algebraic group}

\author{Michel Brion}
\email{Michel.Brion@ujf-grenoble.fr}

\begin{document}
 
\begin{abstract}
Let $G$ be a group scheme of finite type over a field, and consider 
the cohomology ring $H^*(G)$ with coefficients in the structure sheaf. 
We show that $H^*(G)$ is a free module of finite rank over its component 
of degree 0, and is the exterior algebra of its component of degree 1. 
When $G$ is connected, we determine the Hopf algebra structure of $H^*(G)$.
\end{abstract}

\maketitle

\section{Introduction}
\label{sec:intro}

To each scheme $X$ over a field $k$, one associates the 
graded-commutative $k$-algebra
$H^*(X) := \bigoplus_{i \geq 0} H^i(X,\cO_X)$
with multiplication given by the cup product. 
Any morphism of schemes $f : X \to X'$ induces a pull-back homomorphism 
of graded algebras 
$f^* : H^*(X') \to H^*(X)$,
and there are K\"unneth isomorphisms 
$H^*(X) \otimes H^*(Y) \stackrel{\cong}{\longrightarrow} H^*(X \times Y)$. 
When $X$ is affine, the ``coherent cohomology ring'' $H^*(X)$ is just 
the algebra $\cO(X)$ of global sections of $\cO_X$.

\medskip

Now consider a $k$-group scheme $G$ with multiplication map
$\mu : G \times G \to G$, neutral element $e_G \in G(k)$, and 
inverse map $\iota :G \to G$. Then $H^*(G)$ has the structure of 
a graded Hopf algebra with comultiplication $\mu^*$, counit $e_G^*$ 
and antipode $\iota^*$. If $G$ acts on a scheme $X$ and $\cF$ is 
a $G$-linearized quasi-coherent sheaf on $X$, then the cohomology 
$H^*(X,\cF)$ is equipped with the structure of a graded comodule 
over $H^*(G)$.

\medskip

When $G$ is affine, the Hopf algebra $H^*(G) = \cO(G)$ uniquely 
determines the group scheme $G$. But this does not extend to an
arbitrary group scheme $G$; for example, if $G$ is an abelian
variety, then the structure of $H^*(G)$ only depends of $g := \dim(G)$.
Indeed, by a result of Serre (see~\cite[Chap.~7, Thm.~10]{Serre}), 
$H^*(G)$ is the exterior algebra $\Lambda^*(H^1(G))$; moreover, 
$H^1(G)$ has dimension $g$ and consists of the primitive
elements of $H^*(G)$ (recall that $\gamma \in H^*(G)$ is primitive 
if $\mu^*(\gamma) = \gamma \otimes 1 + 1 \otimes \gamma$).

\medskip

In the present article, we generalize this result as follows:

\begin{theorem}\label{thm:main}
Let $G$ be a group scheme of finite type over $k$. Then the graded 
algebra $H^*(G)$ is the exterior algebra of the $\cO(G)$-module
$H^1(G)$, which is free of finite rank.  

If $G$ is connected, then denoting by $P^*(G) \subset H^*(G)$ 
the graded subspace of primitive elements, we have an isomorphism 
of graded Hopf algebras
\[ H^*(G) \cong \cO(G) \otimes \Lambda^*(P^1(G)). \]
Moreover, $P^i(G) = 0$ for all $i \geq 2$.
\end{theorem}

As a consequence, the graded Lie algebra 
$P^*(G)$ equals $P^0(G) \oplus P^1(G)$, and hence is abelian;
also, the vector space $P^1(G)$ is finite-dimensional.
Note that $P^0(G)$ consists of the homomorphisms of group schemes
$G \to \bG_a$; this vector space is finite-dimensional in
characteristic $0$, but not in prime characteristics (already for
$G = \bG_a$). 

\medskip

When $G$ is an abelian variety and $k$ is perfect, the structure 
of $H^*(G)$ follows readily from that of connected  
graded-commutative Hopf algebras (see~\cite[Thm.~6.1]{Borel})
and from the isomorphism of $H^1(G)$ with the Lie algebra of
the dual abelian variety (see~\cite[\S 13, Cor.~3]{Mumford}).
But for an arbitrary group scheme $G$, Theorem~\ref{thm:main} 
is not a direct consequence of general structure results on
Hopf algebras such as those of Cartier-Gabriel-Kostant 
(see~\cite[Thm. 8.1.5]{Sweedler}) 
and Milnor-Moore (see~\cite[\S 6]{MilnorMoore}), since $H^*(G)$ 
is neither connected nor cocommutative. Also, returning to the 
setting of schemes, the $\cO(X)$-module $H^*(X)$ is generally 
far from being free. For example, when $X$ is the punctured affine 
plane, $H^1(X)$ is a torsion module over $\cO(X) = k[x,y]$ and 
is not finitely generated.

\medskip

The proof of Theorem~\ref{thm:main} is based on the affinization 
theorem (see~\cite[Exp. VIB, Thm.~12.2]{SGA3}).
It asserts that $G$ has a smallest normal subgroup scheme $H$ 
such that the quotient $G/H$ is affine; then $\cO(G/H) \cong \cO(G)$ 
via the quotient morphism $G \to G/H$, which is therefore 
identified with the canonical morphism $G \to \Spec \cO(G)$. 
In particular, the $k$-algebra $\cO(G)$ is finitely generated.
Moreover, $H$ is smooth, connected and contained in the center 
of the neutral component $G^o$; in particular, $H$ is commutative. 
Also, we have $\cO(H) = k$, i.e., $H$ is ``anti-affine''. 
In fact, $H$ is the largest anti-affine subgroup scheme of $G$; 
we denote it by $G_{\ant}$.

\medskip

By analyzing the quotient morphism $G \to G/G_{\ant}$, we obtain
an isomorphism of $\cO(G)$-modules
$\psi : H^*(G) \stackrel{\cong}{\longrightarrow} 
\cO(G) \otimes H^*(G_{\ant})$
which identifies the pull-back $H^*(G) \to H^*(G_{\ant})$ to 
$e_G^* \otimes \id$ (Proposition~\ref{prop:mod}). 
On the other hand, using the structure of anti-affine groups
(see~\cite{Br09, Sancho}) and additional arguments, we show that
the Hopf algebra $H^*(G_{\ant})$ is the exterior algebra of 
$H^1(G_{\ant})$, a finite-dimensional vector space 
(Corollary~\ref{cor:pos} and Proposition~\ref{prop:zero}).
This implies the first assertion of Theorem~\ref{thm:main}.

\medskip

When $G$ is connected, we show that the above map $\psi$ is
an isomorphism of graded Hopf algebras (Proposition~\ref{prop:hopf});
moreover, $P^1(G) \cong H^1(G_{\ant})$ via pull-back. 
This yields a description of the primitive elements 
which takes very different forms in characteristic $0$ and in positive 
characteristics. We refer to Theorem~\ref{thm:pri} for the full 
statement, and mention a rather unexpected consequence: 
\emph{in positive characteristics, the group schemes $G$ such that 
$H^*(G) = k$ are trivial; in characteristic $0$, they are exactly the
fibered products $S \times_A E$, where $S$ is an anti-affine extension
of an abelian variety $A$ by a torus, and $E$ is the universal vector
extension of $A$} (Corollary~\ref{cor:acy}).

\medskip

A natural problem is to describe the coherent cohomology ring of 
group schemes over (say) discrete valuation rings. In this setting,
a version of the affinization theorem is known
(see~\cite[Exp.~VIB, Prop.~12.10]{SGA3}), but the structure 
of ``anti-affine'' group schemes is an open question.

\medskip

This article is organized as follows. Section~\ref{sec:ls} collects
preliminary results on linearized sheaves which should be well-known,
but which we could not locate in the literature under a form suited
for our purposes. In Section~\ref{sec:rag}, we show how to reduce 
the structure of $H^*(G)$ to that of $H^*(G_{\ant})$. The latter is 
determined in Section~\ref{sec:caag}, and our main results 
(Theorems~\ref{thm:main} and~\ref{thm:pri}) are proved in 
Section~\ref{sec:fin} by putting everything together.

\bigskip 

\noindent
{\bf Acknowledgements.} I thank the Tata Institute of Fundamental Research
and the Chennai Mathematical Institute for support, and B.~Conrad, 
O.~Gabber, S.~Kumar, G.~Laumon, C.~Teleman, J.~Weyman for very helpful 
discussions or e-mail exchanges.

\section{Linearized sheaves}
\label{sec:ls}

Throughout this article, we consider schemes and their morphisms over 
a fixed field $k$. Unless otherwise mentioned, schemes are assumed
to be separated and of finite type. We use \cite{SGA3} as a general
reference for group schemes, and fix such a group scheme $G$.

We begin by recalling some notions on actions of group schemes
(see~\cite[Exp.~I, \S 6]{SGA3}).
A $G$-\emph{scheme} is a scheme $X$ equipped with a $G$-action
\[ \alpha : G \times X \longrightarrow X, 
\quad (g,x) \longmapsto g \cdot x, \]
i.e., with a morphism of schemes that satisfies the axioms of 
a group action. Given two $G$-schemes $X$, $Y$, a morphism 
$u : X \to Y$ is $G$-\emph{equivariant} if the diagram
\begin{equation}\label{eqn:equi}
\CD
G \times X @>{\id \times u}>> G \times Y \\
@V{\alpha}VV @V{\beta}VV \\
X @>{u}>> Y \\
\endCD
\end{equation}
commutes, where $\beta$ denotes the $G$-action on $Y$; we also
say that $u$ is a $G$-\emph{morphism}.
 
A $G$-\emph{linearization} of a quasi-coherent sheaf $\cF$ on the 
$G$-scheme $X$ is an isomorphism
\begin{equation}\label{eqn:lin}
\Phi : \alpha^*(\cF) \stackrel{\cong}{\longrightarrow} p_2^*(\cF)
\end{equation}
(where $p_2 : G \times X \to X$ denotes the projection) such that 
the following cocycle condition holds: for any (commutative)
$k$-algebra $R$ and for any $g,h \in G(R)$, we have
\begin{equation}\label{eqn:coc}
\Phi_{gh} = \Phi_h \circ h^*(\Phi_g),
\end{equation}
where we denote by 
\[ \Phi_g : g^*(\cF_R) \stackrel{\cong}{\longrightarrow} \cF_R \]
the isomorphism of sheaves over $X_R := \Spec(R) \times X$ 
obtained from $\Phi$ by base change with 
$g \times \id: X_R \to G \times X$. A sheaf equipped with a 
$G$-linearization will be called a $G$-\emph{sheaf}. 

Given two $G$-sheaves $\cF$, $\cG$ on a $G$-scheme $X$, 
a morphism of sheaves of $\cO_X$-modules $\varphi : \cF \to \cG$ 
is $G$-\emph{equivariant}, or a $G$-\emph{morphism}, if the square
\[ \CD
g^*(\cF_R) @>{\Phi_g}>> \cF_R \\
@V{g^*(\varphi_R)}VV @V{\varphi_R}VV \\
g^*(\cG_R) @>{\Psi_g}>> \cG_R \\
\endCD \]
commutes for any $k$-algebra $R$ and any $g \in G(R)$, where $\Phi$
(resp. $\Psi$) denotes the linearization of $\cF$ (resp. $\cG$). 
The $G$-sheaves and $G$-morphisms form an abelian category that
we denote by $\QCoh^G(X)$; the coherent $G$-sheaves are the objects
of a full abelian subcategory, $\Coh^G(X)$. 
By~\cite[Lem.~1.4]{Thomason}, any $G$-sheaf is the direct limit of 
its coherent $G$-subsheaves. 

If $X = \Spec(k)$, then a quasi-coherent sheaf is just a $k$-vector
space $V$, and a $G$-linearization, a linear representation of $G$
on $V$. So $\QCoh^G(X)$ is equivalent to the category $\Mod(G)$ 
of $G$-modules, and $\Coh^G(X)$, to the full subcategory $\modf(G)$
of finite-dimensional $G$-modules. 

We will need the following variant of a result in~\cite[p. 94]{HL10}):

\begin{lemma}\label{lem:nat}
Let $u: X \to Y$ be an equivariant morphism of $G$-schemes, and
$\cF$ (resp. $\cG$) a $G$-sheaf on $X$ (resp. $Y$). Then the
higher direct images $R^i u_*(\cF)$ $(i \geq 0)$, and the pull-back 
$u^*(\cG)$ are equipped with natural structures of $G$-sheaves.
In particular, $H^i(X,\cF)$ is a $G$-module for any $G$-sheaf $\cF$.

If in addition $u$ sits in a cartesian square of equivariant 
morphisms of $G$-schemes,
\[ \CD
X' @>{u'}>> Y' \\
@V{v'}VV @V{v}VV \\
X @>{u}>> Y, 
\endCD \]
where $v$ is flat, then the base change isomorphism
\[ \theta_{u,v} : v^* R^i u_*(\cF) 
\stackrel{\cong}{\longrightarrow} 
R^i u'_* v'^*(\cF) \]
is $G$-equivariant.
\end{lemma}

\begin{proof}
The assertions on the pull-back $u^*(\cG)$ and the direct image
$u_*(\cF)$ are special cases 
of~\cite[Exp.~I, Rem.~6.5.2, Lem.~6.6.1]{SGA3}. The assertion
on higher direct images is checked similarly; we provide additional 
details on base change isomorphisms (treated as equalities 
in~[loc.~cit.]) for completeness.

Since $\beta$ is flat, the cartesian square (\ref{eqn:equi}) 
yields a base change isomorphism
\[ \theta_{u,\beta} : \beta^* R^i u_*(\cF) 
\stackrel{\cong}{\longrightarrow} 
R^i(\id \times u)_* \alpha^*(\cF). \]
We obtain similarly an isomorphism 
\[ \theta_{u,p_2} : p_2^* R^i u_*(\cF) 
\stackrel{\cong}{\longrightarrow} 
R^i(\id \times u)_* p_2^*(\cF). \]
Thus, there is a unique isomorphism 
\[ \Psi : \beta^* R^i u_*(\cF) 
\stackrel{\cong}{\longrightarrow} p_2^* R^iu_*(\cF) \]
such that the square
\[ \CD
\beta^* R^iu_*(\cF) @>{\theta_{u,\beta}}>>  
R^i(\id \times u)_* \alpha^*(\cF) \\
@V{\Psi}VV @V{R^i(\id \times u)_* \alpha^*(\Phi)}VV \\
p_2^* R^iu_*(\cF) @>{\theta_{u,p_2}}>> R^i(\id \times u)_* p_2^*(\cF) \\
\endCD \]
commutes; then for any $g \in G(R)$, the induced morphism
$\Psi_g : g^* R^iu_*(\cF) \to R^iu_*(\cF)$ satisfies
$\Psi_g = R^iu_*(\Phi_g) \circ \theta_{u,g}$. 
To show that $\Psi$ is a $G$-linearization of $R^i u_*(\cF)$, 
it remains to check that
$\Psi_{gh} = \Psi_h \circ h^*(\Psi_g)$ for all $g,h \in G(R)$. Using 
the analogous condition for $\Phi$ and the equality
\[ \theta_{u,gh} = \theta_{u,h} \circ h^*(\theta_{u,g}), \]
this reduces to checking that 
$\theta_{u,h} \circ h^* R^iu_*(\Phi_g) = 
R^iu_* h^*(\Phi_g) \circ \theta_{u,h}$, 
i.e., the square
\[ \CD
h^* R^i u_* g^*(\cF_R)  @>{h^* R^iu_*(\Phi_g)}>>  h^* R^iu_*(\cF_R) \\
@V{\theta_{u,h}}VV  @V{\theta_{u,h}}VV \\
R^iu_* h^*g^*(\cF_R)  @>{R^iu_* h^*(\Phi_g)}>>  R^iu_* h^*(\cF_R) \\
\endCD 
\]
commutes. But this follows from the compatibility of base change
isomorphisms with isomorphisms of sheaves. 

Finally, the assertion on $\theta_{u,v}$ is equivalent to
the commutativity of the square
\[ \CD
g^* v^* R^iu_*(\cF_R) @>>> v^* R^iu_*(\cF_R) \\
@V{g^*(\theta_{u,v})}VV @V{\theta_{u,v}}VV \\
g^* R^iu'_* v'^*(\cF_R) @>>> R^iu'_* v'^*(\cF_R) \\
\endCD \]
for any $g \in G(R)$, where the horizontal maps are induced by the
linearizations. Using the equivariance of $u,v,u',v'$,
this amounts to the commutativity of the square
\[ \CD
v^* R^iu_*(g^*\cF_R) @>{v^* R^iu_*(\Phi_g)}>> v^* R^iu_*(\cF_R) \\
@V{\theta_{u,v}}VV @V{\theta_{u,v}}VV \\
R^iu'_* v'^*(g^*\cF_R) @>{R^iu'_*v'^*(\Phi_g)}>> R^iu'_* v'^*(\cF_R) \\
\endCD \]
which follows again from the compatibility of base change
isomorphisms with isomorphisms of sheaves.
\end{proof}

We also record the following variant 
of~\cite[Exp.~I, Prop.~6.6.2]{SGA3}:

\begin{lemma}\label{lem:com}
Let $\cF$ be a $G$-sheaf on a $G$-scheme $X$. Then $H^*(X,\cF)$
is a graded left Hopf comodule over the graded Hopf algebra $H^*(G)$. 
Moreover, the comodule map
\[ \Delta : H^*(X,\cF) \longrightarrow H^*(G) \otimes H^*(X,\cF) \]
is compatible with the $G$-module structure, i.e., the square
\begin{equation}\label{eqn:com}
\CD
H^*(X,\cF) @>{\Delta}>> H^*(G) \otimes H^*(X,\cF) \\  
@VVV @V{g^* \times \id}VV \\
H^*(X_R,\cF_R) @>{g^*}>> H^*(X_R,\cF_R)  \\
\endCD
\end{equation}
(where the left vertical arrow is the pull-back map)
commutes for any $k$-algebra $R$ and any $g \in G(R)$.
\end{lemma}

We omit the proof which follows similar lines as that of 
Lemma~\ref{lem:nat}, the functorial properties of base change
isomorphisms being replaced with those of K\"unneth isomorphisms.

Note that the component of bi-degree $(0,i)$, 
$\Delta^{(0,i)} : H^i(X,\cF) \to \cO(G) \otimes H^i(X,\cF)$,
is the comodule map for the $G$-module structure of 
$H^i(X,\cF)$; the image of $\Delta^{(0,i)}$ is the subspace of
$G$-invariants, where $G$ acts on $\cO(G) \otimes H^i(X,\cF)$
via its action on $\cO(G)$ by right multiplication, 
and its action on $H^i(X,\cF)$ defined in Lemma~\ref{lem:nat}.

We now turn to the behavior of linearized sheaves under torsors,
and obtain a slight generalization of~\cite[Thm. 4.2.14]{HL10}:

\begin{lemma}\label{lem:tor}
Let $H$ be a group scheme, $X$ a $G \times H$-scheme, $Y$ an $H$-scheme,
and $u : X \to Y$ a $G$-torsor which is also $H$-equivariant. Then
the pull-back $u^*$ and the invariant direct image $u_*^G$ yield 
equivalences of categories 
\[ \QCoh^H(Y) \cong \QCoh^{G \times H}(X), \quad 
\Coh^H(Y) \cong \Coh^{G \times H}(X). \]
\end{lemma}

\begin{proof}
Consider first the case where $H$ is trivial. Then a $G$-linearization 
of a quasi-coherent sheaf $\cF$ is just a descent data for the faithfully 
flat morphism $u$. Thus, $u^*$ is an equivalence from 
$\QCoh^H(Y)$ to $\QCoh^{G \times H}(X)$, 
by~\cite[Exp. VIII, Cor. 1.3]{SGA1}; it restricts to an equivalence
from $\Coh^H(Y)$ to $\Coh^{G \times H}(X)$ by~[loc. cit., Rem. 1.12]. 
Moreover, the natural map
$\cF \to u_*^G u^*(\cF)$ is an isomorphism, as follows 
from~[loc. cit., Cor. 1.7]. This proves the assertion in this case.

In the general case, note that $u^*(\cG)$ is a $G \times H$-sheaf on $X$ 
for any $H$-sheaf $\cG$ on $Y$, by Lemma~\ref{lem:nat} applied to the
$G \times H$-equivariant morphism $u$ (relative to the trivial action 
of $G$ on $Y$). Conversely, for any $G \times H$-sheaf $\cF$ on $X$, 
the data of its $H$-linearization descends to an $H$-linearization 
of $u_*^G(\cF)$, as follows from descent for morphisms of 
sheaves (see~[loc. cit., Cor. 1.2]).
\end{proof}

Next, consider a subgroup scheme $H$ of $G$ and an $H$-scheme $Y$.
Then $H$ acts freely on $G \times Y$ via 
$h \cdot (g,y) := (gh^{-1}, hy)$; we assume that the quotient 
$X := (G \times Y)/H$ is a scheme, and denote it by $G \times^H Y$. 
We have a $G$-action on $X$ via left multiplication on $G$, and a
cartesian square of $G$-morphisms
\begin{equation}\label{eqn:ass} 
\xymatrix{
G \times Y \ar[r]^{p_1}\ar[d]_r & G \ar[d]^q \\
X \ar[r]^f & G/H, \\}
\end{equation}
where $p_1$ denotes the projection, and $q$, $r$ the quotients by $H$. 
The fiber of $f$ at the base point of $G/H$ is identified to $Y$; let 
\[ j : Y \longrightarrow X \]
be the corresponding closed immersion, and
\[ p_2 : G \times Y \longrightarrow Y \]
the projection. We may now state the following variant 
of~\cite[Lem. 1.3]{Thomason}:

\begin{lemma}\label{lem:ass}
With the above notation, the pull-back $j^*$ and the composition
$r_*^H \circ p_2^*$ yield equivalences 
$\QCoh^G(X) \cong \QCoh^H(Y)$ and $\Coh^G(X) \cong \Coh^H(Y)$.
\end{lemma}
 
\begin{proof}
Applying Lemma~\ref{lem:tor} to the $G$-equivariant $H$-torsor 
$r : G \times Y \to X$ and to the $H$-equivariant $G$-torsor
$p_2 : G \times Y \to Y$, we obtain that $r_*^H \circ p_2^*$ is 
an equivalence from $\QCoh^H(Y)$ to $\QCoh^G(X)$,
and likewise for coherent sheaves. 
We now show that $j^*$ yields the inverse equivalence.
Note that $j = r \circ s$, where 
\[ s : Y \longrightarrow G \times Y, \quad y \longmapsto (e_G,y) \] 
is a section of $p_2$. Let $\cF$ be a $G$-sheaf on $X$; then 
$\cF = r_*^H p_2^*(\cG)$ for a unique $H$-sheaf $\cG$ on $Y$. 
Thus, $j^*(\cF) = s^* r^* r_*^H p_2^*(\cG) = s^* p_2^*(\cG) = \cG$.
\end{proof}

In particular, if $Y = \Spec(k)$ then $X = G/H$; also, recall that
an $H$-sheaf on $Y$ is just an $H$-module. Thus, Lemma~\ref{lem:ass}
yields familiar equivalences of categories 
\[ \QCoh^G(G/H) \cong \Mod(H), \quad \Coh^G(G/H) \cong \modf(H). \]
For any $H$-module $M$, we denote by 
\[ \cL_{G/H}(M) := q_*^H(\cO_G \otimes M) \]
the associated $G$-sheaf on $G/H$, where we recall that 
$q : G \to G/H$ stands for the quotient morphism.

Returning to an arbitrary $H$-scheme $Y$, we obtain a variant 
of~\cite[5.19]{Jantzen}:

\begin{lemma}\label{lem:coh}
With the above notation, let $\cF$ be a $G$-sheaf on 
$X = G \times^H Y$, and $\cG := j^*(\cF)$ the corresponding $H$-sheaf 
on $Y$. Then for any $i \geq 0$, there is an isomorphism of $G$-sheaves 
\[ R^if_*(\cF) \cong \cL_{G/H}(H^i(Y,\cG)). \]
\end{lemma}

\begin{proof}
The cartesian square (\ref{eqn:ass}), where $q$ is flat, yields a 
pull-back isomorphism
\[ q^* R^i f_*(\cF) \stackrel{\cong}{\longrightarrow} 
R^i (p_1)_* r^*(\cF)\]
which is a morphism of $G \times H$-sheaves in view of 
Lemmas~\ref{lem:nat} and~\ref{lem:tor}. 
But $\cF \cong r_*^Hp_2^*(\cG)$ by Lemma~\ref{lem:ass}, and hence
$r^*(\cF) \cong p_2^*(\cG)$ by Lemma~\ref{lem:tor} again. Thus,
\[ q^* R^i f_*(\cF) \cong R^i (p_1)_* p_2^*(\cG)
\cong \cO_G \otimes H^i(Y,\cG). \]
This yields the required isomorphism
$R^i f_*(\cF) \cong q_*^H(\cO_G \otimes H^i(Y,\cG))$.
\end{proof}

\section{Reduction to an anti-affine group}
\label{sec:rag}

We still consider a group scheme $G$ of finite type over $k$. 
The action of $G \times G$ on $G$ 
via left and right multiplication: $(x,y) \cdot z := x z y^{-1}$, 
equips $\cO_G$ with the structure of a $G \times G$-sheaf. By 
Lemmas~\ref{lem:nat} and~\ref{lem:com}, this defines a structure of 
$G \times G$-module on $H^*(G)$ which is compatible with
its structure of graded Hopf algebra; in particular, with its
structure of $\cO(G)$-module. We now describe the 
$(G \times G)$-$\cO(G)$-module $H^*(G)$ in terms of the largest 
anti-affine subgroup $G_{\ant}$, by carefully keeping track of 
all the actions:

\begin{proposition}\label{prop:coh}
For each integer $i \geq 0$, there is an isomorphism
\begin{equation}\label{eqn:iso}
\varphi_i : H^i(G) \stackrel{\cong}{\longrightarrow}
(\cO(G \times G) \otimes H^i(G_{\ant}))^G,
\end{equation}
where the right-hand side denotes the subspace of $G$-invariants
for the action of $G$ on $\cO(G \times G) \otimes H^i(G_{\ant})$
via its action on $G \times G$ by right multiplication
$(t \cdot (z,w) = (zt^{-1}, wt^{-1}))$, and its action
on $G_{\ant}$ by conjugation. 

Moreover, $\varphi_i$ is a morphism of 
$(G \times G)$-$\cO(G)$-modules, where $\cO(G)$ acts on the 
right-hand side via the algebra homomorphism
\[ (\id \times \iota^*) \circ \mu^* : \cO(G) \longrightarrow 
\cO(G \times G), \quad 
f \longmapsto ((z,w) \mapsto f(zw^{-1})), \]
and $G \times G$ acts via its action on $G \times G$ by left
multiplication: $(x,y) \cdot (z,w) = (xz,yw)$.

Finally, the square
\begin{equation}\label{eqn:coh}
\CD
H^i(G)  @>{\varphi_i}>> (\cO(G \times G) \otimes H^i(G_{\ant}))^G \\
@V{j^*}VV @V{\delta^* \otimes \id}VV \\
H^i(G_{\ant})  @>{\Delta}>> (\cO(G) \otimes H^i(G_{\ant}))^G \\
\endCD
\end{equation}
commutes, where $j : G_{\ant} \to G$ denotes the inclusion, 
$\delta: G \to G \times G$ the diagonal (so that 
$\delta^* : \cO(G \times G) = \cO(G) \otimes \cO(G) \to \cO(G)$ 
is the multiplication), and $\Delta$ the comodule map
for the $G$-module structure of $H^i(G_{\ant})$.
\end{proposition}

\begin{proof}
We identify the $G \times G$-scheme $G$ to the quotient 
$(G \times G)/\delta(G)$; similarly, the $G \times G$-scheme
$G/G_{\ant}$ is identified to $(G \times G)/K$, where 
\[ K := \delta(G) (G_{\ant} \times G_{\ant}) 
= \delta(G) (e_G \times G_{\ant}). \]
Thus, $K$ is isomorphic to the semi-direct product $G_{\ant} \ltimes G$, 
where $G$ acts on $G_{\ant}$ by conjugation. Moreover, the quotient 
morphism $G \to G/G_{\ant}$ is identified to the natural morphism
\[ f : (G \times G)/\delta(G)  = (G \times G) \times^K  K/\delta(G)
\cong (G \times G) \times^K G_{\ant} \longrightarrow (G \times G)/K, \]
where $K$ acts on $G_{\ant}$ via the action of $G_{\ant}$ on itself
by multiplication, and the action of $G$ by conjugation. 
In view of Lemma~\ref{lem:coh}, this yields 
an isomorphism of $G \times G$-linearized sheaves of 
$\cO_{G/G_{\ant}}$-modules:
\[ R^i f_*(\cO_G) \cong q_*^K(\cO_{G\times G} \otimes H^i(G_{\ant})), \]
where $q : G \times G \to (G \times G)/K$ denotes the quotient morphism,
and $K$ acts on $H^i(G_{\ant})$ via its above action on $G_{\ant}$. 
Since $G/G_{\ant}$ is affine, we obtain an isomorphism of 
$(G \times G)$-$\cO(G/G_{\ant})$-modules 
\[ H^i(G) \cong (\cO(G \times G) \otimes H^i(G_{\ant}))^K. \]
But $\cO(G/G_{\ant}) \cong \cO(G)$; moreover, the anti-affine group 
$G_{\ant}$ acts trivially on its module 
$\cO(G \times G) \otimes H^i(G_{\ant})$. This yields the isomorphism
(\ref{eqn:iso}).

To show the final assertion, note that $j^* : H^i(G) \to H^i(G_{\ant})$
factors as the natural map
\[ H^i(G) \longrightarrow 
H^i(G) \otimes_{\cO(G)} k = H^i(G) \otimes_{\cO(G/G_{\ant})} k \] 
(associated to $e_G^*: \cO(G) \to k$, or equivalently, to    
$e_{G/G_{\ant}}^*: \cO(G/G_{\ant}) \to k$), followed by the map 
\begin{equation}\label{eqn:ten}
H^i(G) \otimes_{\cO(G/G_{\ant})} k \longrightarrow H^i(G_{\ant})
\end{equation}
obtained by base change in the cartesian square
\[ \CD
G_{\ant} @>>> \Spec(k) \\
@V{j}VV @V{e_{G/G_{\ant}}}VV \\
G @>{f}>> G/G_{\ant}.\\
\endCD \]
But this base change map yields an isomorphism 
$e_{G/G_{\ant}}^* R^if_*(\cO_G) \stackrel{\cong}{\to} H^i(G_{\ant})$
by Lemma~\ref{lem:coh}, and hence (\ref{eqn:ten}) is an isomorphism
as well. This identifies $j^*$ to $\otimes_{\cO(G)} k$ 
for the structure of $\cO(G)$-module on 
$(\cO(G \times G) \otimes H^i(G_{\ant}))^G$, and 
implies the commutativity of (\ref{eqn:coh}) by using the fact 
that $\delta(G)$ is the fiber at $e_G$ of the morphism 
$\mu \circ (\id \times \iota): G \times G \to G$, 
$(z,w) \mapsto z w^{-1}$.
\end{proof}

The above result entails a description of the comultiplication 
$\mu^* : H^*(G) \longrightarrow H^*(G) \otimes H^*(G)$,  
since the latter is the comodule map for the $G$-action on 
$H^*(G)$ via left multiplication on $G$. Likewise, the counit and
antipode may be described via the isomorphisms (\ref{eqn:iso}).
But the resulting determination of the Hopf algebra structure of
$H^*(G)$ is very indirect, and will not be developed here.

We now obtain a simpler version of Proposition~\ref{prop:coh}, 
where the left and right $G$-actions take different forms:

\begin{proposition}\label{prop:mod}
For each integer $i \geq 0$, there is an isomorphism of 
$(G \times G)$-$\cO(G)$-modules
\[ \psi_i: H^i(G) \stackrel{\cong}{\longrightarrow} 
\cO(G) \otimes H^i(G_{\ant}), \] 
where $G \times e_G$ (resp. $e_G \times G$) acts on the right-hand side
via its action on $G$ by left multiplication (resp. its action on $G$
by right multiplication, and on $H^i(G_{\ant})$ by conjugation
on $G_{\ant}$), and $\cO(G)$ acts by multiplication on itself. Moreover, 
the triangle
\begin{displaymath}
\xymatrix{
H^i(G)  \ar[r]\ar[dr]_{j^*}  & 
 \cO(G) \otimes H^i(G_{\ant}) \ar[d]^{e_G^* \otimes \id} \\
  &  H^i(G_{\ant}) \\
}
\end{displaymath}
commutes.

In particular, there are isomorphisms of vector spaces
\[ H^i(G)^{G \times e_G} \cong H^i(G)^{e_G \times G} 
\cong H^i(G_{\ant}), \quad 
H^i(G)^{G \times G} \cong H^i(G_{\ant})^G. \]
\end{proposition}

\begin{proof}
Consider the automorphism $u$ of $G \times G$ given by
$u(z,w) = (zw^{-1},w)$; then 
$u((x,y) \cdot (z,w) \cdot t) = (x z w^{-1} y^{-1}, y w t^{-1})$.
The induced automorphism $u^*$ of 
$\cO(G \times G) = \cO(G) \otimes \cO(G)$ 
yields an isomorphism for any $G$-module $M$:
\[ u^* \otimes \id : (\cO(G) \otimes \cO(G) \otimes M)^G 
\stackrel{\cong}{\longrightarrow} 
\cO(G) \otimes ( \cO(G) \otimes M)^G, \]
where $G$ acts on $\cO(G) \otimes \cO(G) \otimes M$ 
(resp. on $\cO(G) \otimes M$) via its right action on 
$G \times G$ (resp. on $G$) and its given action on $M$. 

Also, recall the isomorphism of $G$-modules
\[ e_G^* \otimes \id : (\cO(G) \otimes M)^G 
\stackrel{\cong}{\longrightarrow} M, \]
where the invariants in the left-hand side are taken for
the $G$-action on $\cO(G)$ via right multiplication, and 
the given $G$-action on $M$ (this isomorphism is the inverse
of the comodule map 
$\Delta : M \stackrel{\cong}{\longrightarrow} (\cO(G)\otimes M)^G$).
Thus, we obtain an isomorphism
\[ (\cO(G \times G) \otimes M)^G 
\stackrel{\cong}{\longrightarrow}
\cO(G) \otimes M, \]
which is also $\cO(G)$-linear for the action of $\cO(G)$ on
$\cO(G) \otimes \cO(G)$ via $(\id \times \iota^*) \circ \mu^*$,
and $G \times G$-equivariant for the action on the left-hand side
via left multiplication on $G \times G$, and on the right-hand side 
as in the statement. 

Taking $M = H^i(G_{\ant})$ and applying Proposition~\ref{prop:coh},
we obtain the isomorphism $\psi_i$ and its compatibility properties.
The assertions on invariants follow readily.
\end{proof}

The $G \times G$-invariants in $H^*(G)$ are related to the 
primitive elements as follows: 

\begin{proposition}\label{prop:pri}
For each integer $i \geq 1$, the space of homogeneous primitive elements 
of degree $i$ satisfies 
\[ P^i(G) \subset H^i(G)^{G \times G}, \]
with equality for $i = 1$. 
\end{proposition}

\begin{proof}
Recall that for any $G$-module $M$, we have 
$M^G = \{ m \in M ~\vert~ \Delta(m) = 1 \otimes m \}$,
where $\Delta : M \to \cO(G) \otimes M$ denotes the comodule map. 
Together with Lemma~\ref{lem:com}, it follows that
\[ H^i(G)^{G \times e_G} = \{ \gamma \in H^i(G) ~\vert~ 
\mu^{(0,i)}(\gamma) = 1 \otimes \gamma \}, \]
where $\mu^{(0,i)} : H^i(G) \to \cO(G) \otimes H^i(G)$ 
denotes the component of bi-degree $(0,i)$ of the comultiplication
\[ \mu^* : H^i(G) \to H^i(G \times G) 
= \bigoplus_{i_1, i_2 ; i_1 + i_2 = i} H^{i_1}(G) \otimes H^{i_2}(G). \]
Likewise, 
\[ H^i(G)^{e_G \times G} = \{ \gamma \in H^i(G) ~\vert~ 
\mu^{(i,0)}(\gamma) = \gamma \otimes 1 \}. \]
As a consequence,
\[ H^i(G)^{G \times G} = \{ \gamma \in H^i(G) ~\vert~ 
\mu^*(\gamma) - \gamma \otimes 1 - 1 \otimes \gamma 
\in \bigoplus_{i_1, i_2 > 0; i_1 + i_2 = i} 
H^{i_1}(G) \otimes H^{i_2}(G)\}. \]
This yields our statement.
\end{proof}

Next, we obtain a refinement of Proposition~\ref{prop:mod}:

\begin{proposition}\label{prop:hopf}
If $G$ is connected, then we have equalities of subspaces 
of $H^*(G)$:
\[ H^*(G)^{G \times e_G} = H^*(G)^{e_G \times G} = H^*(G)^{G \times G} \]
and this subspace is a graded Hopf subalgebra, isomorphic to 
$H^*(G_{\ant})$ via $j^*$. Moreover, we have an isomorphism of 
graded Hopf algebras
\[ H^*(G) \cong \cO(G) \otimes H^*(G)^{G\times G}. \]
\end{proposition}

\begin{proof}
By Proposition~\ref{prop:mod}, we may identify the $G \times G$-module
$H^*(G)$ to $\cO(G) \otimes H^*(G_{\ant})$, where $G \times G$
acts on $\cO(G)$ via left and right multiplication, and $G \times G$
acts trivially on $H^*(G_{\ant})$ (since $G_{\ant}$ is central in $G$ 
in view of the connectedness assumption). Then
\[ H^*(G)^{G \times e_G} = H^*(G)^{e_G \times G} = 
1 \otimes H^*(G_{\ant}), \]
since $\cO(G)^{G \times e_G} = \cO(G)^{e_G \times G} = k$. This proves
the equalities. 

The comultiplication
\[ \mu^* : H^*(G) \longrightarrow H^*(G) \otimes H^*(G) \] 
sends $H^*(G)^{G \times G}$ to 
$H^*(G)^{G \times e_G} \otimes H^*(G)^{e_G \times G}$, since the
multiplication $\mu : G \times G \to G$ is equivariant for the action 
of $G \times G$ on itself via $(x,y) \cdot (z,w) = (xz, w y^{-1})$,
and for the action of $G \times G$ on $G$ via left and right 
multiplication. Thus, $H^*(G)^{G \times G}$ is a subcoalgebra of 
$H^*(G)$. Likewise, the antipode $\iota^*: H^*(G) \to H^*(G)$ preserves 
$H^*(G)^{G \times G}$. Hence $H^*(G)^{G \times G}$ is a Hopf subalgebra; 
it is mapped isomorphically to $H^*(G_{\ant})$ by the morphism 
of Hopf algebras $j^*$, in view of Proposition~\ref{prop:mod}. 

For the final assertion, note that the multiplication of $H^*(G)$
induces an isomorphism 
$\cO(G) \otimes (1 \otimes H^*(G_{\ant})) 
\stackrel{\cong}{\longrightarrow} H^*(G),$
by Proposition~\ref{prop:mod} again.
\end{proof}

\begin{remark}\label{rem:nc}
When $G$ is not connected, the three subspaces 
$H^*(G)^{G \times e_G}$, $H^*(G)^{e_G \times G}$ and 
$H^*(G)^{G \times G}$ of $H^*(G)$ are generally distinct.

For example, let $E$ be an elliptic curve, and $G$ the 
semi-direct product of $E$ with the group of order $2$ acting 
on $E$ via multiplication by $\pm 1$. Then $G_{\ant} = E$ and 
$H^*(G_{\ant})^G = k$. 
Thus, $H^1(G)^{G \times G} = 0$, while $H^1(G)^{G \times e_G}$ and 
$H^1(G)^{e_G \times G}$ are two distinct copies of $k$ in 
$H^1(G) \cong k^2$. In particular, $P^1(G) = 0$.
\end{remark}

\section{The cohomology algebra of an anti-affine group}
\label{sec:caag}

Throughout this section, we assume that $G$ is anti-affine,
i.e., $\cO(G) = k$.
Recall from~\cite[\S 2]{Br09} that $G$ sits in a unique extension
of smooth connected commutative group schemes,
\[ 0 \longrightarrow T \times U \longrightarrow G 
\stackrel{\alpha}{\longrightarrow} A \longrightarrow 0, \]
where $A$ is an abelian variety, $T$ is a torus, and $U$ is 
unipotent; moreover, $U$ is trivial if $\charc(k) > 0$. 
Thus, we obtain two extensions
\[ 0 \longrightarrow T \longrightarrow G/U 
\stackrel{\alpha_T}{\longrightarrow} A \longrightarrow 0, \]
\[ 0 \longrightarrow U \longrightarrow G/T 
\stackrel{\alpha_U}{\longrightarrow} A \longrightarrow 0, \]
where $G/U$ and $G/T$ are anti-affine as well. Note that $G/U$ is a
semi-abelian variety, and $G/T$ an extension of an abelian variety
by a vector group; also, we have an isomorphism of group schemes
\begin{equation}\label{eqn:str}
G \cong G/U \times_A G/T.
\end{equation}

This yields a useful reduction to the case where $T$ is trivial:

\begin{proposition}\label{prop:iso}
With the above notation, the pull-back under the quotient morphism
$G \to G/T$ yields an isomorphism of graded Hopf algebras
$H^*(G/T) \cong H^*(G)$.
\end{proposition}

\begin{proof}
We may replace $k$ with any field extension, and hence assume that
the torus $T$ is split.

Since the morphism $\alpha: G \to A$ (the quotient by $T \times U$) 
is affine, we have
\begin{equation}\label{eqn:aff}
H^*(G) = H^*(G,\cO_G) = H^*(A,\alpha_*(\cO_G)). 
\end{equation}
Moreover, (\ref{eqn:str}) yields an isomorphism
\[ \alpha_*(\cO_G) \cong 
(\alpha_T)_*(\cO_{G/U}) \otimes_{\cO_A} (\alpha_U)_*(\cO_{G/T}). \]
So it suffices to show that the map 
\[ H^*(A, (\alpha_U)_*(\cO_{G/T})) \longrightarrow
H^*(A, (\alpha_T)_*(\cO_{G/U}) \otimes_{\cO_A} (\alpha_U)_*(\cO_{G/T})) \]
induced by the pull-back $\cO_A \to (\alpha_T)_*(\cO_{G/U})$, 
is an isomorphism.

By~\cite[\S 2.1]{Br09}, there is an isomorphism of sheaves of 
$\cO_A$-modules
\begin{equation}\label{eqn:dec}
(\alpha_T)_*(\cO_{G/U}) \cong \bigoplus_{\lambda \in \wT} \cL_{\lambda},
\end{equation} 
where $\wT$ denotes the character group of $T$, and each 
$\cL_{\lambda}$ is an algebraically trivial invertible sheaf on $A$.
Moreover, $\cL_0 \cong \cO_A$ but $\cL_{\lambda}$ is non-trivial for
any $\lambda \neq 0$. In view of~\cite[\S 8, p. 76]{Mumford}, 
it follows that 
\begin{equation}\label{eqn:van}
H^*(A,\cL_{\lambda}) = 0 \quad (\lambda \neq 0). 
\end{equation}

On the other hand, since $\alpha_U$ is a torsor under the 
vector group $U$, the sheaf $(\alpha_U)_*(\cO_{G/T})$ has an increasing 
filtration by coherent subsheaves indexed by the non-negative 
integers, with subquotients being the structure sheaf $\cO_A$.
(Indeed, $(\alpha_U)_*(\cO_{G/T}) = \cL_{(G/T)/U}(\cO(U))$, and
the $U$-module $\cO(U)$ has an increasing filtration with 
subquotients being the trivial module $k$). By (\ref{eqn:van}), 
this yields
\[ H^*(A,\cL_{\lambda} \otimes_{\cO_A}(\alpha_U)_*(\cO_{G/T})) = 0 
\quad (\lambda \neq 0). \] 
Together with (\ref{eqn:dec}), this completes the proof.
\end{proof}

\begin{corollary}\label{cor:pos}
If $\charc(k) >0$, then the pull-back $\alpha^*: H^*(A) \to H^*(G)$
is an isomorphism of graded Hopf algebras.
\end{corollary}

In view of these results, we may assume that $\charc(k) = 0$,
and $G$ is an extension of the abelian variety $A$ by the vector 
group $U$. Recall that there is a universal such extension,
$0 \to V \to E \to A \to 0$, where $V := H^1(A)^{\vee}$ (the dual
vector space of $H^1(A)$, viewed as an additive group). Moreover, 
by~\cite[\S 2.2]{Br09}, $E$ is anti-affine and we have a commuting 
diagram of extensions
\begin{equation}\label{eqn:ext}
\xymatrix{0 \ar[r] & V \ar[r]\ar[d]^{\gamma} & 
E \ar[r]^{\beta} \ar[d] & A \ar[r] \ar[d]^{\id} & 0\\
0 \ar[r] & U \ar[r] & G \ar[r]^{\alpha} & A \ar[r] & 0, \\}
\end{equation}
where the classifying map $\gamma$ is surjective. 

\begin{proposition}\label{prop:zero}
With the above notation and assumptions, the homomorphism 
$\alpha^*: H^*(A) \to H^*(G)$ is surjective, and its kernel is 
the ideal of $H^*(A) = \Lambda^*(H^1(A)) = \Lambda^*(V^{\vee})$ 
generated by the image of 
$\gamma^{\vee} : U^{\vee} \to V^{\vee}$ (the transpose of 
$\gamma : V \to U$). In particular, we have an isomorphism 
of graded Hopf algebras $H^*(G) \cong \Lambda^*(W^{\vee})$, where 
$W:= \Kern(\gamma)$.
\end{proposition}

\begin{proof}
We argue by induction on $\dim(U)$. If $U = 0$, then $G = A$
and the statement is obvious. So we assume that $U \neq 0$, and
choose a non-zero $u \in U$. This yields an exact sequence
of vector groups 
\[ \CD
0 @>>> \bG_a @>>> U @>{\varphi}>> U' @>>> 0,
\endCD \]
where $1 \in \bG_a$ is sent to $u$. We also obtain a derivation 
$D$ of the $k$-algebra $\cO(U)$, given by 
$D(f) := \frac{d}{dt} f(x + t u)\vert_{t = 0}$.
Equivalently, $D$ is the vector field associated with $u$ viewed 
as a point of the Lie algebra of $U$. Then $D$ is surjective and
its kernel is $\cO(U)^{\bG_a} \cong \cO(U')$; in other words, we have
an exact sequence of $U$-modules
\begin{equation}\label{eqn:ind} 
\CD
0 @>>> \cO(U') @>{\varphi^*}>> \cO(U) @>{D}>> \cO(U) @>>> 0.
\endCD 
\end{equation}
Next, let $G' := G/\bG_a$ so that $G'$ sits in two extensions
\[ \CD 0 \to \bG_a @>>> G @>{\varphi}>> G' \to 0, \quad
0 \to U' @>>> G' @>{\alpha'}>> A \to 0. \endCD \]
By~(\ref{eqn:aff}), we have 
$H^*(G) = H^*(A,\alpha_*(\cO_G))$; likewise, 
$H^*(G') = H^*(A,\alpha'_*(\cO_{G'})$. Moreover, $\alpha_*(\cO_G)$
(resp. $\alpha'_*(\cO_{G'})$ is the $G$-sheaf on $A = G/U$ 
associated to the $U$-module $\cO(U)$ (resp. $\cO(U')$).
But the exact sequence (\ref{eqn:ind}) yields an exact sequence 
of $G$-sheaves
\[ \CD
0 @>>> \alpha'_*(\cO_{G'}) @>{\varphi^*}>> \alpha_*(\cO_G) 
@>{D}>> \alpha_*(\cO_G) @>>> 0,
\endCD \]
and hence a long exact sequence of cohomology groups
\[ \CD
\cdots H^{i-1}(G) @>{D}>> H^{i-1}(G) @>>> H^i(G')
@>{\varphi^*}>> H^i(G) @>{D}>> H^i(G) \cdots
\endCD \] 
Now $D$ acts on $H^*(G)$ via the action of the Lie algebra of
$G$ arising from the $G$-action by (left or right) multiplication.
But the anti-affine group $G$ acts trivially on its module
$H^*(G)$, and hence $D$ acts by $0$. This yields short exact sequences
\begin{equation}\label{eqn:rec} 
\CD 
0 @>>> H^{i-1}(G) @>>> H^i(G') @>{\varphi^*}>> H^i(G) @>>> 0.
\endCD 
\end{equation}
In particular, $\varphi^* : H^*(G') \to H^*(G)$ is surjective.
Using the induction assumption, it follows that the natural 
homomorphism $\psi: \Lambda^*(H^1(G)) \to H^*(G)$ is surjective. 
On the other hand, by (\ref{eqn:rec}), the Poincar\'e polynomial
$P_{H^*(G)}(t) := \sum_{i \geq 0} \dim(H^i(G)) \, t^i$ satisfies
$P_{H^*(G')}(t) = (1 + t) P_{H^*(G)}(t)$. By the induction assumption 
again, this yields $P_{H^*(G)}(t) = (1 + t)^n$, where 
$n = \dim(H^1(G')) - 1 = \dim(H^1(G))$. Thus, $H^*(G)$ and 
$\Lambda^*(H^1(G))$ have the same Poincar\'e polynomial; hence
$\psi$ is an isomorphism.

To complete the proof, it remains to construct an isomorphism
$W^{\vee} \stackrel{\cong}{\to} H^1(G)$, compatible with pull-backs
of anti-affine extensions of $A$ by vector groups.
We do this in two steps.

First, we construct an isomorphism 
$W^{\vee} \stackrel{\cong}{\longrightarrow} \Ext^1(G,\bG_a)$
compatible with such pull-backs. For this, consider the 
exact sequence of commutative group schemes
$ 0 \to U \to G \to A \to 0$
and the associated long exact sequence
\[
0 \to \Hom(A,\bG_a) \to \Hom(G,\bG_a) \to \Hom(U,\bG_a) \to \]
\[ \to \Ext^1(A,\bG_a) \to \Ext^1(G,\bG_a) \to \Ext^1(U,\bG_a). \]
We have $\Hom(A,\bG_a) = 0 = \Hom(G,\bG_a)$, since $G$ is 
anti-affine; also, $\Hom(U,\bG_a) = U^{\vee}$, 
$\Ext^1(A,\bG_a) = V^{\vee}$ and $\Ext^1(U,\bG_a) = 0$. Moreover,
the pushout map $\partial : \Hom(U,\bG_a) \to \Ext^1(A,\bG_a)$ 
is identified to the transpose of the classifying map
$\gamma : V \to U$; thus, 
$\Coker(\partial) \cong \Coker(\gamma^{\vee}) = W^{\vee}$. This 
yields the required isomorphism.

Next, we show that the natural map
\[ u : \Ext^1(G,\bG_a) \to H^1(G,\cO_G) = H^1(G) \]
that associates to each extension the class of the corresponding
$\bG_a$-torsor, is an isomorphism. For this, we argue again by 
induction on $\dim(U)$. If $U = 0$, then $G = A$ and the assertion
is exactly \cite[Chap.~7, Thm.~7]{Serre}. For an arbitrary $U$,
let $W'$, $G'$ be as above; then the exact sequence 
$0 \to \bG_a \to G \to G' \to 0$ yields a commutative diagram 
of exact sequences
\[ \CD
0 @>>> k @>>>  \Ext^1(G',\bG_a) @>{\varphi^*}>> 
\Ext^1(G,\bG_a)  @>>> 0 \\
& & @V{\id}VV @V{u'}VV   @V{u}VV \\
0 @>>> k @>>>  H^1(G') @>{\varphi^*}>>  H^1(G) @>>> 0, \\
\endCD \]
since $\Hom(G',\bG_a) = \Hom(G,\bG_a) = \Hom(\bG_a,\bG_a) = k$
and $\Ext^1(\bG_a,\bG_a) = 0$; the bottom exact sequence is 
(\ref{eqn:rec}) for $i = 1$. By the induction assumption,
$u'$ is an isomorphism; it follows that so is $u$.
\end{proof}

\begin{remark}\label{rem:alt}
We present an alternative proof of Proposition~\ref{prop:zero}
which is more conceptual but less self-contained. We first claim that
\begin{equation}\label{eqn:acy}
H^i(E) = 0 \quad (i > 0).
\end{equation}
This has been proved by G.~Laumon in an unpublished preprint 
(see~\cite[Thm.~2.4.1]{Laumon}); we provide another argument 
as follows.

Since the morphism $\beta : E \to A = E/V$ is affine, we have
$H^i(E) = H^i(A,\beta_*(\cO_E))$. Moreover, $\beta_*(\cO_E)$ 
is the $E$-sheaf on $A$ corresponding to the $V$-module 
$\cO(V)$. The latter may be characterized as the injective hull of 
the trivial module $k$ (the unique simple module).

Now recall that the category $\modf(V)$ of finite-dimensional 
$V$-modules is equivalent to $\Coh^E(A)$, via
$M \mapsto \cL_{E/V}(M)$. Moreover, each coherent $E$-sheaf $\cF$ 
on $A$ has a finite increasing filtration with subquotients being 
the structure sheaf $\cO_A$, i.e., $\cF$ is the sheaf of local
sections of a \emph{unipotent} vector bundle. In fact, this yields an 
equivalence from $\Coh^E(A)$ to the category $\Uni(A)$ of unipotent 
vector bundles on $A$ (see~\cite[Rem.~3.13(ii)]{Br12}).
Also, recall that $\Uni(A)$ is equivalent to the category 
$\Coh_0(\wA)$ of coherent sheaves on the dual abelian variety 
$\wA$ supported at the origin, via the Fourier-Mukai transform
that assigns to a coherent sheaf on $\wA$ supported at $0$, the
sheaf $(p_1)_*(\cP \otimes_{\cO_{\wA}} p_2^*(\cF))$ on $A$; here
$p_1,p_2$ denote the projections from $A \times \wA$, 
and $\cP$ stands for the Poincar\'e bundle on $A \times \wA$ 
(see~\cite[Thm. 4.12]{Mukai}).

Thus, we obtain an equivalence of abelian categories from $\modf(V)$ 
to $\Coh_0(\wA)$. By taking direct limits, this extends to an equivalence 
of abelian categories $F: \Mod(V) \to \QCoh_0(\wA)$. 
Thus, $F$ sends $\cO(V)$ to the injective hull $I$ of the residue 
field $k(0)$. By~\cite[Thm. 4.12]{Mukai} again, we have 
\[ H^i(A,\cL_{E/V}(M)) \cong \Ext^i_{\cO_{\wA,0}}(k(0),F(M)) \]
for any finite-dimensional $V$-module $M$ and any $i \geq 0$.
Since cohomology commutes with direct limits, it follows that
\[ H^i(A,\beta_*(\cO_E)) = H^i(A,\cL_{E/V}(\cO(V)))
\cong \Ext^i_{\cO_{\wA,0}}(k(0), F(\cO(V)))
= \Ext^i_{\cO_{\wA,0}}(k(0), I). \]
But the latter vanishes for any $i > 0$; this yields (\ref{eqn:acy}).

Next, recall that $H^*(G) = H^*(A,\alpha_*(\cO_G))$, 
where $\alpha_*(\cO_G)$ is the $E$-sheaf on $A = E/V$ 
corresponding to the $V$-module $\cO(V/W)$.
The Koszul complex yields a resolution of this $V$-module,
\[ 0 \to \cO(V) \otimes \Lambda^n(W^{\vee}) 
\to \cO(V) \otimes \Lambda^{n-1}(W^{\vee}) \to \cdots \to
\cO(V) \to \cO(V/W) \to 0, \]
where $n := \dim(W)$, and hence an exact sequence
\[ 0 \to \beta_*(\cO_E) \otimes \Lambda^n(W^{\vee}) 
\to \beta_*(\cO_E) \otimes \Lambda^{n-1}(W^{\vee}) \to \cdots \to
\beta_*(\cO_E) \to \alpha_*(\cO_G) \to 0. \]
Moreover, $H^i(A,\beta_*(\cO_E))$ vanishes for all $i > 0$, by
(\ref{eqn:acy}). 
So we obtain an isomorphism
\[  H^*(G) = H^*(G,\alpha_*(\cO_A)) \cong \Lambda^*(W^{\vee}). \]
Also, note that $W^{\vee} = V^{\vee}/\Ima(\gamma) = H^1(A)/\Ima(\gamma)$.
So, when $U$ is trivial, we recover the isomorphism 
$H^*(A) \cong \Lambda^* (H^1(A))$; for an arbitrary $U$, we 
obtain the required isomorphism.
\end{remark}

\section{The main results}
\label{sec:fin}

\subsection{Proof of Theorem~\ref{thm:main}}

Recall that $G$ denotes a group scheme of finite type over $k$,
and $j : G_{\ant} \to G$ the inclusion of the largest anti-affine 
subgroup. 

By Proposition~\ref{prop:mod}, there is an isomorphism of 
graded $\cO(G)$-modules 
$\psi : H^*(G) \to \cO(G) \otimes H^*(G_{\ant})$
which identifies $j^* : H^*(G) \to H^*(G_{\ant})$ with 
$e_G^* \otimes \id$. Moreover, by Corollary~\ref{cor:pos}
(when $\charc(k) > 0$) and Proposition~\ref{prop:zero} (when 
$\charc(k) = 0$), the natural map 
$\Lambda^* (H^1(G_{\ant})) \to H^*(G)$ is an isomorphism, and
the $k$-vector space $H^1(G_{\ant})$ is finite-dimensional. 
Thus, $H^*(G)$ is free of finite rank as a graded module over 
$\cO(G)$; recall that the latter algebra is finitely generated. 
Moreover, the natural homomorphism of graded $\cO(G)$-modules
\[ \varphi : \Lambda^*_{\cO(G)} (H^1(G)) \to H^*(G) \]
is an isomorphism at the $k$-point $e_G$. Since $\varphi$ is 
$G \times G$-equivariant, it is an isomorphism everywhere.

If $G$ is connected, then we have isomorphisms of graded Hopf algebras
\[ H^*(G) \cong \cO(G) \otimes H^*(G)^{G \times G} 
\cong \cO(G) \otimes H^*(G_{\ant}) \]
by Proposition~\ref{prop:hopf}. Since 
$P^i(G) \subset H^i(G)^{G \times G}$ for $i \geq 1$ 
(Proposition~\ref{prop:pri}), we obtain that 
$P^i(G) \subset P^i(G_{\ant})$ via pull-back. But in view of the 
structure of $H^*(G_{\ant})$, we have $P^1(G_{\ant}) = H^1(G_{\ant})$
and $P^i(G_{\ant}) = 0$ for $i \geq 2$. Thus, $P^i(G) = 0$ for 
$i \geq 2$ as well. Moreover, by Proposition~\ref{prop:pri} again,
$P^1(G) = H^1(G)^{G \times G}$ and hence we obtain an isomorphism 
\begin{equation}\label{eqn:prim}
j^* : P^1(G) \stackrel{\cong}{\longrightarrow} H^1(G_{\ant}). 
\end{equation}
This completes the proof of Theorem~\ref{thm:main}.

\medskip

Combining Proposition~\ref{prop:hopf}, Corollary~\ref{cor:pos} and 
Proposition~\ref{prop:zero}, we also obtain the characterization 
of `acyclic' group schemes mentioned in the introduction:

\begin{corollary}\label{cor:acy}
Let $G$ be a group scheme of finite type over $k$. Then
$H^*(G) = k$ if and only if $G$ is trivial when $\charc(k) > 0$,
resp. $G \cong S \times_A E$ when $\charc(k) = 0$, where $S$ is 
an anti-affine extension of an abelian variety $A$ by a torus, and
$E$ is the universal vector extension of $A$.
\end{corollary}

\subsection{The primitive elements}

From now on, we assume that $G$ is connected; we will describe 
the space $P^1(G)$ of homogeneous primitive elements of degree $1$,
in terms of the structure of $G$. 

By~\cite[Lem. IX 2.7] {Raynaud}, $G$ has a normal connected linear
subgroup scheme $L$ such that the quotient $G/L$ is an abelian 
variety; we denote by
\[ \alpha : G \longrightarrow A := G/L \]
the quotient homomorphism. In characteristic $0$, one easily
sees that $L$ is the largest connected linear (or, equivalently,
affine) subgroup scheme of $G$. In particular, $L$ is unique; 
we then set $L := G_{\aff}$ and $A := A(G)$. This does not extend 
to positive characteristics, since we may replace $A$ with its
quotient by any infinitesimal subgroup scheme. Yet when $G$ is
smooth, there exists a (unique) smallest subgroup scheme $L$
as above (see~\cite[Thm. 9.2.1]{BLR}); we denote again $L$ by
$G_{\aff}$, and $G/L$ by $A(G)$. For example, if $G$ is 
anti-affine, then $G_{\aff} = T \times U$ with the notation 
of Section~\ref{sec:caag}. Returning to an arbitrary connected 
group scheme $G$, the largest abelian quotient $A$ is related to 
$A(G_{\ant})$ as follows:

\begin{proposition}\label{prop:ros}
With the above notation and assumptions, we have a commutative square
\begin{equation}\label{eqn:ros}
\xymatrix{
G_{\ant} \ar[r]^{j}\ar[d]_{\alpha_{\ant}} & G \ar[d]^{\alpha} \\
A(G_{\ant}) \ar[r]^{\varphi} & A, \\
} 
\end{equation}
where $\varphi$ is an isogeny.
\end{proposition}

\begin{proof}
The homomorphism $\alpha \circ j : G_{\ant} \to A$ sends
the smooth connected affine group scheme 
$T \times U = \Kern(\alpha_{\ant})$ to the origin of the abelian variety
$A$, and hence factors through $\alpha_{\ant}$. This shows the existence 
of $\varphi$. The quotient of $A$ by the
image of $\alpha \circ j$ is connected and proper, but also
affine as a quotient group scheme of $G/G_{\ant}$. Thus, 
$\alpha \circ j$ is surjective, and hence so is $\varphi$.
Finally, since $\Kern(\alpha \circ j) = L \cap G_{\ant}$, we obtain
an isomorphism 
$\Kern(\varphi) \cong (L \cap G_{\ant})/\Kern(\alpha_{\ant}) 
= (L \cap G_{\ant})/(G_{\ant})_{\aff}$.
The latter quotient is a linear subgroup scheme of the abelian
variety $G_{\ant}/(G_{\ant})_{\aff} = A(G_{\ant})$,
and hence is finite. Thus, $\varphi$ is an isogeny.
\end{proof}

\begin{theorem}\label{thm:pri}
Keep the above notation. If $\charc(k) > 0$, then 
$\dim(P^1(G)) = \dim(A)$. If $\charc(k) = 0$, then 
$P^1(G)$ is the image of $\alpha^* : H^1(A(G)) \to H^1(G)$.
\end{theorem}

\begin{proof}
If $\charc(k) >0$, then  
$\dim(P^1(G)) = \dim (H^1(G_{\ant})) = \dim (H^1(A(G_{\ant})))$
by (\ref{eqn:prim}) and Corollary~\ref{cor:pos}. Thus, 
$\dim (P^1(G)) = \dim (A(G_{\ant}))$.
But $\dim(A(G_{\ant})) = \dim (A)$ by Proposition~\ref{prop:ros}.
This proves the first assertion.

For the second assertion, note that $\Ima(\alpha^*)$ is contained 
in $P^1(G)$. Also, the commutative square (\ref{eqn:ros}) yields 
a commutative square of pull-backs
\begin{equation}\label{eqn:H1}
\xymatrix{
H^1(A(G)) \ar[r]^{\varphi^*}\ar[d]_{\alpha^*} & 
H^1(A(G_{\ant})) \ar[d]^{\alpha^*_{\ant}} \\
P^1(G) \ar[r]^{j^*} & H^1(G_{\ant}), \\
}
\end{equation}
where $j^*$ is the isomorphism (\ref{eqn:prim}); moreover, 
$\alpha^*_{\ant}$ is surjective in view of Proposition~\ref{prop:zero}.
Thus, it suffices to show that $\varphi^*$ is an isomorphism. 
But since $\varphi$ is an isogeny, there exists an isogeny 
$\psi : A(G) \to A(G_{\ant})$ such that $\varphi \circ \psi$ 
is the multiplication $n_{A(G)}$ for some positive integer $n$, 
and $\psi \circ \varphi = n_{A(G_{\ant})}$. Moreover, 
$n_{A(G)}^* : H^1(A(G)) \to H^1(A(G))$ is just multiplication by $n$,
and similarly for $n^*_{A(G_{\ant})}$ (as follows e.g. from the
isomorphism $H^1(A,\cO_A) \cong \Ext^1(A,\bG_a)$
for any abelian variety $A$, and from the bilinearity of $\Ext^1$). 
It follows that $\varphi$ is indeed an isomorphism.
\end{proof}

When $\charc(k) > 0$, it may happen that 
$\alpha^* : H^1(A(G)) \to H^1(G)$ is zero while $P^1(G) \neq 0$, 
as shown by the following:

\begin{example}\label{exam:fin}
Let $p := \charc(k)$ and let $E$ be an elliptic curve 
such that the $p$-torsion subgroup scheme $E_p$ 
(the kernel of the multiplication $p_E$) is isomorphic to 
$\bZ/p \bZ \times \bmu_p$, where $\mu_p$ denotes the multiplicative
group of $p$-th roots of unity; this holds if $E$ is ordinary and 
the base field $k$ is sufficiently large. 
We may view $E_p$ as a subgroup scheme of $\bG_a \times \bG_m$; then
\[ G := E \times^{E_p} (\bG_a \times \bG_m) \]
is a smooth connected commutative group scheme. We have
\[ G_{\aff} = E_p \times^{E_p}(\bG_a \times \bG_m) 
\cong \bG_a \times \bG_m \]
and $A(G) = E/E_p \cong E$; moreover, 
$G_{\ant} = E \times^{E_p} E_p \cong E$. Hence $G_{\ant} = A(G_{\ant})$, 
and the isogeny $\varphi : A(G_{\ant}) \to A(G)$ is identified to
the quotient morphism $E \to E/E_p$, that is, to $p_E : E \to E$.
Thus, $\varphi^* : H^1(E) \to H^1(E)$ is zero. In view of the 
commutative diagram (\ref{eqn:H1}), where $j^*$ is an isomorphism and 
$\alpha_{\ant}^* = \id$, it follows that $\alpha^* = 0$. But
$P^1(G) \cong H^1(G_{\ant}) \cong k$.
\end{example}

\end{document}